\documentclass[reqno]{amsart}

\numberwithin{equation}{section}
\usepackage{latexsym}
\usepackage{amsmath}
\usepackage{amssymb}
\usepackage{mathrsfs}
\usepackage{graphicx,colordvi}
\usepackage{upgreek}
\usepackage{ifthen}
\usepackage{color}
\usepackage{hyperref}

\usepackage[T1]{fontenc}
\usepackage[latin1]{inputenc}

\numberwithin{equation}{section}

\newtheorem{defi}{Definition}[section]
\newtheorem{theorem}[defi]{Theorem}
\newtheorem{lemma}[defi]{Lemma}

\newtheorem{proposition}[defi]{Proposition}
\newtheorem{remark}[defi]{Remark}
\newtheorem{remarks}[defi]{Remarks}

\usepackage{latexsym}
\usepackage{amsmath}
\usepackage{amssymb}

\newcommand{\cE}{{\mathcal E}}
\newcommand{\cH}{{\mathcal H}}
\newcommand{\cB}{{\mathcal B}}

\newcommand{\C}{{\mathbb C}}

\newcommand{\EE}{{\mathbb E}}

\newcommand{\N}{{\mathbb N}}
\newcommand{\II}{{\mathbb I}}

\newcommand{\R}{{\mathbb R}}

\newcommand{\PP}{{\mathbb P}}

\providecommand{\norm}[1]{\left\vert #1 \right\vert}

\renewcommand{\epsilon}{\varepsilon}

\newcommand{\sa}{{\sf a}}
\newcommand{\bb}{{\bb}}

\begin{document}

\title[Motion of a rigid body with a fluid-filled cavity]{On the motion of a fluid-filled rigid body with Navier Boundary conditions}

\author{}
\date{\today}
\author{Giusy Mazzone}
\address{Department of Mathematics\\
        Vanderbilt University\\
        Nashville, Tennessee\\l
        USA}
\email{giusy.mazzone@vanderbilt.edu}

\author{Jan Pr\"uss}
\address{Martin-Luther-Universit\"at Halle-Witten\-berg\\
         Institut f\"ur Mathematik \\
         Theodor-Lieser-Strasse 5\\
         D-06120 Halle, Germany}
\email{jan.pruess@mathematik.uni-halle.de}

\author{Gieri Simonett}
\address{Department of Mathematics\\
        Vanderbilt University\\
        Nashville, Tennessee\\
        USA}
\email{gieri.simonett@vanderbilt.edu}
\thanks{This work was supported by a grant from the Simons Foundation (\#426729, Gieri Simonett).}



\maketitle

\begin{abstract}
We consider the inertial motion of a system constituted by a rigid body with an interior cavity entirely filled with a viscous incompressible fluid. Navier boundary conditions are imposed on the cavity surface. We prove the existence of weak solutions and determine the critical spaces for the governing evolution equation. Using parabolic regularization in time-weighted spaces, we establish regularity of solutions and their long-time behavior. We show that every weak solution \`a la Leray-Hopf to the equations of motion converges to an equilibrium at an exponential rate in the  $L_q$-topology for every fluid-solid configuration.  

A nonlinear stability analysis shows that equilibria associated with the largest moment of inertia  
are asymptotically (exponentially) stable, whereas all other equilibria are normally hyperbolic and unstable in an appropriate topology. \\

\end{abstract}

\maketitle

\section{Introduction and formulation of the problem}
In this paper, we investigate the stability properties and long-time behavior of the system $\mathcal S$ constituted by a rigid body with a cavity completely filled by a viscous incompressible fluid. No external forces act on the system, and so 
the motion is driven by the inertia of the whole system once an initial angular momentum is imparted on $\mathcal S$ ({\em inertial motion}). We suppose that the fluid is subject to Navier boundary conditions on the cavity surface. Under these conditions (see \eqref{problem}$_{3,4}$ below), the fluid is allowed to partially slip on
the impermeable  boundary of the cavity.  More precisely, it is  assumed that the fluid normal velocity is zero, whereas the slip velocity is proportional to the shear stress on the solid boundary (\cite{Na1827}). 
 
 Problems of this kind arise in the study of fluid-solid interactions in 
many fields of the applied sciences, including microfluidics, geophysics, and cardiovascular science (see \cite{YbBaCoJoBo07,VaDoRa2007,Mi09,MuCa16} and the references contained therein).

From the mathematical point of view, so far there have been only few contributions aimed at furnishing a rigorous treatment of the problem at hand. Major results concerning the study of viscous incompressible fluids subject to Navier boundary conditions either consider the case where the fluid is confined to a fixed immovable domain  (\cite{Be05,ChQi10,NePe18,PrSiWi18,PrWi17}), or  the case where the fluid is flowing around moving (rigid or elastic) structures (\cite{NePe10,GeHi14,PlSu14,MuCa16,ChNe17}). 

In contrast, there is a large literature dealing with the motion of fluid-filled rigid bodies with {\em no-slip} boundary conditions,
it spans from the early work by Stokes \cite{Sto}, Zhukovskii \cite{Zh}, Hough \cite{Ho}, Poincar\'e \cite{Po}, and Sobolev \cite{So} to more recent contributions mostly concerned with stability problems (\cite{Ru3,Ru4,MoRu,Ch,Lya,KoShYu,KoKr00,SiTa,Ma12,GaMaZuCRM,DGMZ16,Ma16}). A comprehensive study of the motion of fluid-filled rigid bodies has recently been given in \cite{Ga17} (in an $L_2$ framework),
 and in \cite{MaPrSi18} (in a more general $L_q$ framework).
 It is shown that equilibria correspond to permanent rotations (rotations with constant angular velocity around the central axes of inertia) of $\mathcal S$ with the fluid at a relative rest with respect to the solid. Moreover, equilibria associated with the largest moment of inertia are stable, while all other equilibria are unstable. Finally, it has been proved that every Leray-Hopf weak solution converges to an equilibrium at an exponential rate in an appropriate topology. 
 
These results have a natural physical explanation. Due to  viscosity effects, after a sufficiently large time (which depends on the initial motion imparted on the system as well as the physical properties of the fluid) the fluid goes to a rest state relative to the solid. Thus, the long-time behavior of the coupled system is characterized by a rigid body motion with the system moving as a whole solid. The rigid body dynamics is quite rich, it includes permanent rotations, precessions, and more generally motions \'a la Poinsot. However, once the fluid is at rest, the pressure gradient must balance the centrifugal forces. Thus, the only rigid body motion the coupled system eventually performs is a permanent rotation. The aim of our work is to show that this ``stabilizing effect'' of the fluid on the motion of the rigid body can also be observed for the case of a fluid-filled rigid body with Navier boundary conditions.  

In this manuscript, we answer fundamental mathematical questions about the existence of weak ({\em \'a la Leray-Hopf}) and strong solutions to the equations of motion for the coupled system $\mathcal S$. We show that equilibrium configurations corresponding to {\em permanent rotations} of $\mathcal S$ around the central axes of inertia corresponding to the largest moment of inertia are asymptotically (actually, exponentially) stable. All other permanent rotations are unstable. Our stability results are all obtained in an $L_q$-framework.   
The main ingredient that  allows us to obtain such a result relies on the fact that the 
set of equilibria $\mathcal E$ (locally) forms a finite dimensional manifold, with dimension $m=1,2,3,$ depending on the mass distribution of $\mathcal S$ (see Proposition \ref{prop:equilibria}(c)). Moreover, a detailed study of the spectrum of the linearization  around a nontrivial equilibrium shows that equilibria associated with the largest moment of inertia are normally stable, while all other equilibria are normally hyperbolic (see Theorem \ref{th:spectrum}). These results yield a complete description of the long-time behavior of $\mathcal S$. 

The equations of motion given in \eqref{problem} form a coupled system of nonlinear parabolic PDEs and ODEs with bilinear nonlinearities. The mathematical model features a combination of {\em conservative} and {\em dissipative} properties as can be observed by the {\em conservation of angular momentum} \eqref{momentum-conserved} and the {\em energy inequality} \eqref{eq:strong_energy}. 
These are distinctive characteristics for this  type of fluid-solid interactions (see also \cite{DGMZ16,Ma16,Ga17,MaPrSi18}). We prove the existence of weak solutions \'a la Leray-Hopf to \eqref{problem} corresponding to initial data with arbitrary (finite) kinetic energy. Although the proof adapts standard tools (see e.g. \cite{ChQi10} for the classical Navier-Stokes equations in bounded domains), it nevertheless yields the first result in the literature concerning existence of global weak solutions for the fluid-solid interaction problem \eqref{problem}. In Theorem~\ref{th:local_strong}, we determine the largest space of initial data for which the equations of motion are well-posed. The functional setting we use is that of maximal $L_p-L_q$ regularity in time-weighted $L_p$ spaces (\cite{PrSi16}). Using parabolic regularization in time-weighted spaces (\cite{PrSi16,PrSiWi18}), we show that for any initial data with finite energy, all corresponding weak solutions will converge to an equilibrium at an exponential rate, in the topology of $H^{2\alpha}_q$ with $\alpha \in [0,1)$ and $q\in (1,6)$. The latter is shown by proving that, for each Leray-Hopf solution, there exists a time $\tau>0$ at which the solution has gained sufficient regularity to serve as initial condition for a strong solution constructed in Theorem~\ref{th:local_strong}. 
Well-posedness in $[\tau, \infty)$ then follows by choosing a suitable time-weight which ensures that weak and strong solutions coincide 
and trajectories are relatively compact. It is worth emphasizing that all the above results hold for every fluid-solid configuration. In particular, no assumption has been made on the geometric shape of the body and/or mass distribution within $\mathcal S$. 

Throughout the paper, the friction constant  $\zeta$ in \eqref{problem}$_4$ is assumed to be strictly positive. In fact, in the case 
$\zeta=0$, a preliminary analysis shows that the decay of the fluid's relative velocity can be obtained either in case the fluid domain is not 
a solid of revolution, or else in topological spaces corresponding to fields orthogonal to rigid body motions around the axis of symmetry. 
This is a consequence of Lemma~\ref{lem:FK}. Indeed, in the case $\zeta=0$, the dissipation in the energy balance (see Lemma \ref{lem:energy_balance})  would be only in terms of $|D(v)|_{L_2(\Omega)}$. The dissipative term then  would not be enough to control the norm $|v|_{L_2(\Omega)}$, as no information is given for the tangential component of $v$ to the boundary (see e.g. \cite{SoSc73} for the classical Navier-Stokes equations). 

Here is the plan of our paper. Below, we provide the mathematical formulation of the problem. In Section \ref{se:energy and equilibria}, we introduce the available energy of our system and characterize the set of equilibria.  
The well-posedness of the governing equations and the relevant critical spaces are discussed in 
Section~\ref{se:local well-posedness and critical spaces}. The existence of Leray-Hopf solutions and related properties are proved in Section~\ref{sec:weak}. In Section~\ref{se:long-time}, we analyze the spectrum of the linearization  at an equilibrium, and provide a complete characterization of the nontrivial equilibria as either normally stable or normally hyperbolic. The stability properties of our equilibria, at the nonlinear level, is then analyzed.  We finally use all information gained in the previous sections to provide an exhaustive  description of the long-time behavior of $\mathcal S$. \\

The equations governing the motion of the system fluid-filled rigid body $\mathcal S$ in a {\em non-inertial frame} with origin at $G$
(the center of mass of the whole system), and axes ${\sf e}_i$, $i=1,2,3$, directed along the principal axes of inertia of $\mathcal S$, read as follows: 
\begin{equation}
\label{problem}
\begin{aligned}
\partial_t v + v\cdot\nabla v +(\dot\sa -\dot\omega)\times x + 2(\sa-\omega)\times v
-\upnu \Delta v  +\nabla p & =0 &&\text{in} \;\; \Omega\times\R_+, \\
{\rm div}\, v &= 0 &&\text{in} \;\; \Omega\times\R_+, \\
(v|n) &=0 &&\text{on} \;\; \Gamma\times\R_+, \\
2\upnu \,P_\Gamma(D(v)n)+\zeta v&=0 &&\text{on} \;\; \Gamma\times\R_+, \\
\II\, \dot\sa + (\sa-\omega)\times\II\sa &=0 &&\text{on}\;\; \R^3\times\R_+,\\
(v(0),\sa(0))= (v_0, &\sa_0)  &&\text{in} \;\; \Omega\times \R^3.\\
\end{aligned}
\end{equation}
Here  $\Omega\subset \R^3$ is a bounded domain with boundary $\Gamma=\partial\Omega$ of class $C^{3}$, $v$ denotes the fluid velocity relative to the rigid body, $\upnu$ is its coefficient of kinematic viscosity, and $p$ the pressure field. 
Moreover, $\II={\rm diag}[\lambda_1,\lambda_2,\lambda_3]$ denotes the inertia tensor of $\mathcal S$ with respect to $G$, $\lambda_j$
 are the central moments of inertia of $\mathcal S$, $\II\sa $ is the total angular momentum of $\mathcal S$ with respect to $G$, and 
\begin{equation}
\label{def-omega}
\omega :=\II^{-1}\int_\Omega x \times v\,dx.
\end{equation}
Without loss of generality, we have set the fluid density  equal to one.
We note that the total angular momentum is a conserved quantity,
\begin{equation}
\label{momentum-conserved}
|\II \sa |= |\II \sf a_0|,
\end{equation}
as can readily be seen by taking the inner product of \eqref{problem}$_{5}$ with $\II\sa$. 
The equations in \eqref{problem} can be derived as in \cite{Ma12} (see also \cite{DGMZ16}). The noticeable difference here are the
{\em Navier boundary conditions} \eqref{problem}$_{3,4}.$ The vector field $n$ denotes the outer unit normal on $\Gamma$ and $P_\Gamma:= I-n\otimes n$ is the orthogonal projection onto the tangent bundle ${\sf T}\Gamma$. Equation \eqref{problem}$_{3}$ is a kinematic condition that describes the impermeability of the boundary surface. The constant $\zeta>0$ in \eqref{problem}$_{4}$ is the coefficient of friction between the fluid and the boundary, and 
\[
D(v)=\frac 12 (\nabla v+\nabla v^T)
\]
is the fluid stretching tensor. Equation \eqref{problem}$_{4}$ represents a {\em partial-slip} boundary condition. Dividing both sides of \eqref{problem}$_4$ by $\zeta$ and taking the limit as $\zeta\to \infty$  one formally obtains the case of {\em no-slip} boundary condition $v=0$ on $\Gamma$ considered in~\cite{MaPrSi18}. 

Throughout this paper, we use the notation $(\cdot|\cdot)$ for the Euclidean inner product in $\R^3$ and $|\cdot|$ for the associated norm. Moreover, for a Banach space $X$, $|\cdot|_{X}$ will denote its norm, and  
$B_X(u_0,r)$  the open ball of radius $r$ centered at $u_0\in X$, with respect to the topology of $X$. 

The set of all linear bounded operators from $X$ to $Y$, with $X, Y$ Banach spaces, is denoted by $\cB(X,Y)$.
If $A:D(A)\subset X\to X$ is a linear operator,  $\sigma(A), N(A), R(A)$ stand for the spectrum, the null-space,
and the range of $A$, respectively.

For $q\in [1,\infty]$, $L_q(D)$ denotes the classical Lebesgue spaces, where  $D$ is either $\Omega$ or $\Gamma$. 
In some of the  proofs (and when the context is clear), $|\cdot|_{D}$ will be used for the $L_2$-norm on $D$,
instead of $|\cdot|_{L_2(D)}$. 

If $k\in \N$, $J\subset \R$ is an open interval and $X$ is a Banach space, then $C^k(J;X)$ (resp. $L_p(J;X)$  and $H^k_p(J;X)$)  represents  the space of all $k$-times continuously differentiable (resp. $L_p$- and $H^k_p$-) 
functions\footnote{We will use the same notation for scalar and vector valued functions.} on $J$ with values in $X$.  
For $1<q<\infty$, we 
denote by
\[
L_{q,\sigma}(\Omega)=\{v\in L_q(\Omega)^3: {\rm div}\, v=0\;\;\text{in}\;\;\Omega,\quad  (v|n)=0\;\;\text{on}\;\;\Gamma\}
\]
the space of all solenoidal vector fields on $\Omega$ and  by $\PP$ the Helmholtz projection of 
$L_q(\Omega)^3$ onto $L_{q,\sigma}(\Omega)$. 
Moreover,  we set
\begin{equation}
\label{H-2-parallel}
{_\parallel H}^2_{q,\sigma}(\Omega)
= \{u\in H^2_{q}(\Omega)^3\cap L_{q,\sigma}(\Omega):\; 2\upnu P_\Gamma(D(v)n)+\zeta u=0\text{ on }\Gamma\}.  
\end{equation}
\section{Energy and equilibria}\label{se:energy and equilibria}
\subsection{Dissipation of energy}
The available energy of system \eqref{problem} is given by
\begin{equation}
\label{energy}
{\sf E}={\sf E}(v,\sa):=\frac{1}{2}\big[ |v|_{L_2(\Omega)}^2 - (\II\omega |\omega) + (\II\sa|\sa)\big].
\end{equation}
From the physical point of view, {\sf E} represents the {\em total kinetic energy} of the system fluid-filled rigid body.  
The above functional is positive definite along the solutions 
to \eqref{problem} thanks to the following  result proved in \cite[Sections  7.2.2--7.2.4]{KoKr00} (see also  \cite[Lemma 2.3.3 and the following remarks]{Ma16}).

\begin{lemma}\label{lem:positive_def}
There exists a constant $c\in (0,1]$ such that 
\[
c|v|^2_{L_2(\Omega)}\le |v|^2_{L_2(\Omega)} - (\II\omega |\omega) \le |v|^2_{L_2(\Omega)},
\quad v\in L_2(\Omega)^3. 
\] 
\end{lemma}

Sufficiently smooth solutions to \eqref{problem} enjoy the following energy balance. 
\begin{lemma}\label{lem:energy_balance}
Consider $(v,\sa,p)$, with 
$v\in H^1_2((0,T);L_{2,\sigma}(\Omega))\cap L_2((0,T);   {_\parallel}H^2_{2,\sigma}(\Omega))$, 
$\sa\in H^1_\infty((0,T))$ and $p\in L_2((0,T);H^1_2(\Omega))$,
 satisfying \eqref{problem} a.e. in $\Omega\times (0,T)$ for some $T\in (0,\infty]$. Then,  
$$ \frac{d}{dt}{\sf E}+2\upnu |D(v)|^2_{L_2(\Omega)}+\zeta|v|^2_{L_2(\Gamma)}=0\quad {in}\quad (0,T).$$
\end{lemma}

\medskip
\noindent {\em Proof.}
From \eqref{energy} and \eqref{problem}, one has 
\begin{equation*}
\begin{aligned}
\frac{d}{dt}{\sf E}
& =(\partial_t v|v)_\Omega -(\II\dot\omega|\omega) + (\II\,\dot\sa |\sa) \\
& =\int_\Omega \big( (\upnu\Delta v-\nabla p|v) -(v|v\cdot\nabla v) 
    + ((\dot\omega -\dot\sa)  \times x|v) + 2 ((\omega -\sa)\times v|v)\big)\,dx\\
&\quad -(\II\dot\omega|\omega) +(\II\dot\sa|\sa)\\
& =2\upnu\int_\Omega ({\rm div}D(v)|v) \,dx 
    + (\II(\dot\omega -\dot\sa)|\omega)  -(\II\dot\omega|\omega) +(\II\dot\sa|\sa)\\
& = -\zeta |v|^2_{\Gamma}-2\upnu |D(v)|^2_\Omega +(\II\dot\sa|\sa -\omega)\\
& = -\zeta |v|^2_{\Gamma}-2\upnu |D(v)|^2_\Omega - ((\sa-\omega)\times \II\sa|\sa -\omega)
=-\zeta |v|^2_{\Gamma}-2\upnu |D(v)|^2_\Omega. \qquad\quad{\square}
\end{aligned}
\end{equation*}

\medskip\noindent
The  energy is a {\em strict Lyapunov functional} as the function 
$[t\mapsto {\sf E}(v(t),\sa(t))]$ is strictly decreasing along non-constant solutions.  
In fact, suppose $\frac{d}{dt} {\sf E}=0$ on some interval $(t_1,t_2)$.
Then $D(v)=0$ on $\Omega\times (t_1,t_2)$ and $v=0$ on $\Gamma \times (t_1,t_2)$. 
Hence $v=0$ on $\Omega\times (t_1,t_2)$ 
by Korn's inequality (see for instance \eqref{eq:korn}).
This implies $\omega=0$ on $(t_1,t_2)$, and hence 
$$\dot\sa \times x +\nabla p=0\quad\text{on}\quad (t_1,t_2).$$
Taking the curl on both sides yields
${\rm curl}\, (\dot\sa\times x)=2\dot\sa=0,$
and therefore, we are at an equilibrium. 

\subsection{Equilibria}
Suppose  $(v,\sa)\in {_\parallel H}^2_{q,\sigma}(\Omega)\times\R^3$ is an equilibrium of \eqref{problem}.
Then 
$$2\upnu |D(v)|^2_{L_2(\Omega)}+\zeta|v|^2_{L_2(\Gamma)}=0,$$ and by the argument used above, $v=0$.
Consequently, $\omega=0$ and \eqref{problem}$_5$ then yields
$a\times \II\sa=0.$
This implies  either $\sa = 0$, or $\sa$ and $\II\sa$ must be parallel and therefore,
$\sa\in {\sf N}(\lambda - \II)$ for some $\lambda\in\{\lambda_1,\lambda_2,\lambda_3\}.$ 
This yields  the set of non-trivial equilibria 
\begin{equation}
\label{equilibria1}
\cE= \{(0,\sa_\ast) : a_\ast\in {\sf N}(\lambda_\ast - \II),\;\sa_*\neq 0,\; \lambda_\ast\in\{\lambda_1,\lambda_2,\lambda_3\}\},
\end{equation}
with constant pressure $p$ in each case.
We note, and this will turn out to be important later on, that $\cE$ is locally a manifold 
of dimension $m\in\{1,2,3\}$, with $m$ depending on the distribution of the  central moments of inertia $\lambda_i$.
In more detail, for each non-trivial equilibrium $(0,\sa_*)$, there is a neighborhood  
$U$ in ${_\parallel H}^2_{q,\sigma}(\Omega)\times\R^3$ such that
$\cE\cap U$ is a smooth (in fact flat) manifold.

We summarize our results.
\begin{proposition}\label{prop:equilibria}
The following assertions hold for problem \eqref{problem}.
\begin{enumerate}
\setlength\itemsep{1mm}
\item[(a)] The total angular momentum is conserved.
\item[(b)] The  energy {\sf E}, defined in \eqref{energy}, is a strict Lyapunov functional.
\item[(c)] The set $\cE$ of non-trivial equilibria is given by \eqref{equilibria1}.
$\cE$  is locally a smooth manifold  of dimension $m\in\{1,2,3\}$,  with $m$ depending on the distribution of the  
central moments of inertia $\lambda_i$.
\item[(d)] The critical points of the energy with prescribed total momentum are precisely the equilibria of the system.
\item[(e)] If the energy with prescribed nonzero total momentum has a local minimum at a critical point $(0,\sa_*)$
then necessarily $\lambda_*=\max\{\lambda_1,\lambda_2,\lambda_3\}$.
\end{enumerate}
\end{proposition}
\begin{proof}
The remaining assertions in (d) and (e) can be proved by the method of Lagrange multipliers, 
see also \cite[Proposition 2.3]{MaPrSi18}.
\end{proof}

\section{Local well-posedness and critical spaces}
\label{se:local well-posedness and critical spaces}
\noindent
In this section, we show that system \eqref{problem} is locally well-posed in an $L_q$-setting.
We consider the Banach spaces 
\begin{equation}
\label{X0-X1}
X_0:= L_{q,\sigma}(\Omega)\times \R^3,
\quad X_1:={_\parallel}H^2_{q,\sigma}(\Omega) \times \R^3,\quad 1<q<\infty.
\end{equation}
 Notice that $X_1$ is compactly embedded in $X_0$. 
System   \eqref{problem} can be rewritten in the following equivalent form 
\begin{equation}
\label{problem-2}
\begin{aligned}
\partial_t v + (\dot\sa -\dot\omega)\times x 
-\upnu \Delta v +\nabla p &=f(v,\sa) &&\text{in }  \Omega, \\
{\rm div}\, v &= 0 &&\text{in } \Omega, \\
(v|n) &=0 &&\text{on } \Gamma, \\
2\upnu \,P_\Gamma(D(v)n)+\zeta v&=0 &&\text{on } \Gamma,\\
\II\, \dot\sa &=g(v,\sa) &&\text{on }\R^3,\\
(v(0),\sa(0)) &= (v_0,\sa_0), \\
\end{aligned}
\end{equation}
where  $f(v,\sa)=-v\cdot\nabla v-2(\sa-\omega)\times v$ and $g(v,\sa)=-(\sa -\omega)\times \II \sa.$ 
Let 
\begin{equation}
\label{def-E-A_0}
\begin{aligned}
E(v,\sa) &:=\big(v+\PP\big(x\times \II^{-1}\int_\Omega(x\times v)\,dx) - \PP (x\times \sa), \II\sa\big), 
\quad (v,\sa) \in X_0,\\
A(v,\sa) &:=\big(-\upnu\PP\Delta v , 0\big),\qquad  (v,\sa)\in X_1.
\end{aligned}
\end{equation}
Then problem~\eqref{problem-2} can be reformulated as the following semilinear evolution equation in the unknown $u:=(v,\sa)$
\begin{equation}\label{eq:evolution}
\frac{d}{dt}u +  Lu=F(u),\quad u(0)=u_0.  
\end{equation}
Here $u_0:=(v_0,\sa_0)$, $L:=E^{-1}A$ and $F(u)=(F_1(u), F_2(u)):=E^{-1}(\PP f(v,\sa),g(v,\sa)),$
provided we know that $E$ is invertible. 
Consider the operator 
\begin{equation}
\label{finite-rank}
Kv:=\PP\big(x\times \II^{-1}\int_\Omega(x\times v)\,dx)=\sum_{i=1}^3 \langle \ell_i| v\rangle \PP (x\times {\sf e}_i),
\end{equation}
where $\ell_i$ are bounded linear functionals on $L_{q,\sigma}(\Omega)$.  
One readily verifies that $K$ has finite rank, and then it defines a compact linear operator on $L_{q,\sigma}(\Omega).$ 
In \cite{MaPrSi18}, it has been also proved that $(I+K)$ is invertible on $L_{q,\sigma}(\Omega).$  

Using well-known properties of the cross-product,  \eqref{def-omega} and Lemma~\ref {lem:positive_def}
one verifies that $(Kv|w)_{L_2(\Omega)} =(v|Kw)_{L_2(\Omega)}$ for $v,w\in L_{2,\sigma}(\Omega)$
and 
\begin{equation}
\label{I+K positive}
((I+K)v|v)_{L_2(\Omega)} 
=|v|^2_{L_2(\Omega)}-(\II\omega |\omega)\ge c|v|^2_{L_2(\Omega)},\quad v\in L_{2,\sigma}(\Omega).
\end{equation}
With $I+C:= (I+K)^{-1}$, these remarks lead to the following result (see also \cite{MaPrSi18}). 
\goodbreak
\begin{proposition}
\label{pro:E-invertible}
 $E$ is invertible on $X_0$. The inverse is given by
 \begin{equation*}\label{eq:E-1}
 E^{-1}=\left[\begin{array}{cc}
I+ C & (I+ C) \PP(x\times \II^{-1}\cdot )\\
      0 & \II^{-1}   
\end{array}\right],
 \end{equation*}
where  $I$ is the identity operator on $L_{q,\sigma}(\Omega)$, and where
$C$ has the following properties:
\begin{enumerate}
\setlength\itemsep{1mm}
\item[{\bf (a)}]
 C is a compact (in fact a finite-rank) operator on $L_{q,\sigma}(\Omega)$. 
= \item[{\bf (b)}]
 $I+C$ is invertible on $L_{q,\sigma}(\Omega)$ and  positive definite on $L_{2,\sigma}(\Omega)$. 
\end{enumerate} 
 \end{proposition}
A moment of reflection shows that the operator $L$ in ~\eqref{eq:evolution} is given by
\begin{equation}  
\label{eq:L}
L=\left[\begin{array}{cc}
-\upnu (I+ C)\PP\Delta  &0\\
      0 & 0
      \end{array}\right].
\end{equation} 
Let $A_N$ be the Stokes operator with Navier boundary conditions, given by
\begin{equation}
\label{eq:stokes_domain}
{\sf D}(A_N)={_\parallel H}^2_{q,\sigma}(\Omega),\quad A_Nv= -\PP\Delta v\quad\text{for}\;\; v\in {\sf D}(A_N),
\end{equation}
see \eqref{H-2-parallel} for the definition of ${_\parallel H}^2_{q,\sigma}(\Omega)$. 
Then $A_N, (I+C)A_N: {\sf D}(A_N)\to L_{q,\sigma}(\Omega)$ enjoy the following properties.
\goodbreak
\begin{proposition}
\label{pro:L-calculus}
Let $1<q<\infty$. Then
\begin{enumerate}
\setlength\itemsep{1mm}
\item[{\bf (a)}] $A_N$ is invertible and has a bounded $\cH^\infty$-calculus with $\cH^\infty$-angle $\phi^\infty_{A_N}=0$. 
\item[{\bf (b)}]  $A_1=(I+C)A_N$ is invertible and has a bounded 
$\cH^\infty$-calculus with $\cH^\infty$-angle $\phi^\infty_{A_1}<\pi/2$.
\end{enumerate} 
\end{proposition}
\begin{proof}
(a) As $A_N$ has compact resolvent, $\sigma(-A_N)$, the spectrum of $-A_N$, consists solely of eigenvalues of finite algebraic multiplicity.
Moreover, the spectrum is independent of $q$, and we may restrict ourselves to the case $q=2$.
Suppose that $\mu\in\C$ is an eigenvalue of $-A_N$ with associated eigenfunction $v$, that is, 
$\mu v + A_N v=0$. Taking the $L_2$-inner product of this relation with $\bar v$, the complex conjugate of $v$,
(for questions related to spectral theory we always consider the complexification of $X_0$ and $X_1$) yields
$$0= \mu |v|^2_\Omega - (\PP \Delta v | v)_\Omega= \mu |v|^2_\Omega + |D(v)|^2_\Omega +\zeta |v|^2_\Gamma.$$
This implies $\mu\in (-\infty, \mu_0)$ for some $\mu_0<0$. Therefore, $\sigma(-A_N)\subset (-\infty,\mu_0]$. In particular, 
$A_N$ is invertible.
We refer to \cite[Theorem 4.1]{PrWi17} for the remaining assertion. 

\medskip\noindent
{\bf(b)} 
We will first show that $A_1=(I+C)A_N$ is invertible and sectorial, with sectorial angle $\phi_{A_1}<\pi/2$.
As in the proof of part (a), it suffices to locate the eigenvalues of $A_1$ in $L_{2,\sigma}(\Omega).$

Suppose then that $(\mu +A_1)v=0$ for some $\mu\in\C$ and $v\in {\sf D}(A_N)$.
Applying $(I+C)^{-1}=(I+K)$ to this equation and then taking the $L_2(\Omega)$ inner product with $\bar v$ results in
$
\mu((I+K)v|v)_\Omega - (\PP \Delta v|v)_\Omega = 0.
$
We conclude that $\mu\in (-\infty,\mu_1]$ with $\mu_1<0$, as both $(I+K)$ and $-\PP\Delta$ are positive definite.
In particular, $A_1$ is invertible.

Let $ {\sf A}+ {\sf B}:=-\PP\Delta -C\PP\Delta $  with domain ${\sf D}(A_N)={_\parallel H}^2_{q,\sigma}(\Omega)$. 
By part (a),  ${\sf A}$ admits a bounded $\cH^\infty$-calculus with $\cH^\infty$-angle zero on $L_{q,\sigma}(\Omega)$.
Hence, it is sectorial with angle $\phi_{\sf A}=0$, while $B$ is a relative compact perturbation. It follows
from \cite[Lemma 3.1.7 and Corollary 3.1.6]{PrSi16} that there is a positive number $\eta_0$ such that $\eta_0 +{\sf A} +{\sf B}$ is invertible and  sectorial with spectral angle less than $\pi/2$.
Combining this result with the fact that $\sigma(-A_1)\subset (-\infty,\mu_1]$ implies that
$A_1$ is sectorial with sectorial angle $\phi_{A_1}<\pi/2$ as well.

Next we infer from the representation $C=-K(I+K)^{-1}$ that 
${\sf B}$ maps ${_\parallel H}^2_{q,\sigma}(\Omega)$ into  $C^\infty_\sigma(\bar\Omega)$. 
In particular,  ${\sf B}\in \cB({_\parallel H}^2_{q,\sigma}(\Omega),H^{s}_{q,\sigma}(\Omega))$  for $s\in (0,1/q)$, where the spaces $H^s_{q,\sigma}(\Omega)$ are introduced below. 
From \eqref{Hs-interpolation} and the fact that ${\sf A}$ has bounded imaginary powers, it follows that
\[
H^s_{q,\sigma}(\Omega)=[L_{q,\sigma}(\Omega), {_\parallel H}^2_{q,\sigma}(\Omega)]_{s/2} ={\sf D} (A^{s/2}),
\quad s\in (0,1/q).
\]
Hence ${\sf B}:{\sf D(A)}\to  {\sf D}({\sf A}^\alpha),$ with $\alpha=s/2$, is bounded.
Observing that ${\sf A}+{\sf B}$ is invertible and sectorial, 
we can now follow the proof of \cite[Proposition 3.3.9]{PrSi16} to infer that 
${\sf A} +{\sf B}\in \cH^\infty(L_{q,\sigma}(\Omega))$ 
with $ \cH^\infty$-angle $\phi^\infty_{{\sf A} +{\sf B}}<\pi/2$.
\end{proof}

For $s\in [0,2]$ and $q\in (1,\infty)$ we set 
${H}^s_{q,\sigma}(\Omega)=H^s_q(\Omega)^3\cap L_{q,\sigma}(\Omega)$ and
\[
\begin{split}
&{_\parallel}H^s_{q,\sigma}(\Omega)= 
\left\{\begin{aligned}
&{H}^s_{q,\sigma}(\Omega), && s\in [0,1+1/q),
\\
&[L_{q,\sigma}(\Omega),{_\parallel H}^2_{q,\sigma}(\Omega) ]_{1/2+1/2q},&& s=1+1/q,
\\
&\{v\in H^s_{q,\sigma}(\Omega):\, 2\upnu \,P_\Gamma(D(v)n)+\zeta v=0 \text{ on }\Gamma\},&& s>1+1/q,
\end{aligned}\right.
\end{split}
\]  
where $H^s_q(\Omega)$ denote the Bessel-potential spaces.
Then it holds that 
\begin{equation}
\label{Hs-interpolation}
{_\parallel H}^s_{q,\sigma}(\Omega)=[L_{q,\sigma}(\Omega), {_\parallel H}^2_{q,\sigma}(\Omega)]_{s/2}, \quad s\in (0,2),
\end{equation}
with $[\cdot ,\cdot ]_\theta$ the complex interpolation method, see~\cite[p. 750]{PrWi17}.

Moreover, for $s,q$ as above and $p\in (1,\infty)$, we set
$B^s_{qp,\sigma}(\Omega)=B^s_{qp}(\Omega)^3\cap L_{q,\sigma}(\Omega)$ and 
\begin{equation*}
{_\parallel}B^{s}_{qp,\sigma}(\Omega):=
\left\{ 
\begin{aligned}
&B^s_{qp,\sigma}(\Omega), &&s\in[0,1+1/q),
\\
&(L_{q}(\Omega),{_\parallel H}^2_{q,\sigma}(\Omega))_{1/2+1/2q,p}, &&s=1+1/q,
\\
&\{u\in  B^{s}_{qp,\sigma}(\Omega): 2\upnu \,P_\Gamma(D(v)n)+\zeta v=0 \text{ on }\Gamma\}, &&s>1+1/q,
\end{aligned}
\right.
\end{equation*}
with $B^s_{qp}(\Omega)$ being the Besov spaces.
In this case, the following interpolation result holds
\begin{equation}
\label{Bs-interpolation}
{_\parallel H}^s_{q,\sigma}(\Omega)=(L_{q,\sigma}(\Omega), {_\parallel H}^2_{q,\sigma}(\Omega))_{s/2,p}, \quad s\in (0,2),
\end{equation}
with $(\cdot|\cdot)_{\theta,p}$ the real interpolation method, see~\cite[p. 750]{PrWi17}.
Finally, we recall that
\begin{equation}
\label{BWH}
B^s_{22}(\Omega)= W^s_2(\Omega)=H^s_2(\Omega),
\end{equation}
where $W^s_q(\Omega)$ are the Sobolev-Slobodeckii spaces. 

 Well-posedness of \eqref{eq:evolution} will be established in the following {\em time-weighted spaces}
\begin{equation}
\label{E1-mu}
\EE_{1,\mu}(0,T):=H^1_{p,\mu}((0,T);X_0)\cap L_{p,\mu}((0,T);X_1),
\end{equation}
where, for $T\in (0,\infty)$, $1/p<\mu\le 1$, and $X$ a Banach space,
\begin{equation*}
\begin{aligned}
&u\in L_{p,\mu}((0,T);X) && \Leftrightarrow \quad  t^{1-\mu}u\in L_p((0,T);X),\\
&u\in H^1_{p,\mu}((0,T);X) && \Leftrightarrow \quad u,\dot u\in L_{p,\mu}((0,T);X).
 \end{aligned}
\end{equation*}
We are now ready for our main result on existence and uniqueness of strong solutions for problem \eqref{problem}, or equivalently, 
problem~\eqref{eq:evolution}.
 
\begin{theorem}\label{th:local_strong}
Suppose
\begin{equation}
\label{assumptions-pq}
p\in(1,\infty),\quad q\in (1,3),\quad  2/p +3/q\le 3,
\end{equation}
and let $\mu$ satisfy
\begin{equation}
\label{assumptions-mu}
\mu\in (1/p,1],\quad  \mu\ge \mu_{\rm crit}=\frac{1}{p} + \frac{3}{2q}-\frac{1}{2}. 
\end{equation}
\begin{enumerate}
\setlength\itemsep{1mm}
\item[{\bf (a)}]
Let  $u_0=(v_0,\sa_0)\in {_\parallel}B^{2\mu-2/p}_{qp,\sigma}(\Omega)\times \R^3=X_{\gamma,\mu}$ be given.
Then there are positive constants $T=T(u_0)$ and $\eta=\eta(u_0)$ such that
\eqref{eq:evolution} admits a unique solution $u(\cdot, u_0)=(v,{\sf a})$ in 
 \begin{equation*}
 \qquad
 {\EE_{1,\mu}(0,T)}=H^1_{p,\mu}((0,T); L_{q,\sigma}(\Omega)\times\R^3)
 \cap L_{p,\mu}((0,T); {_\parallel H}^2_{q,\sigma}(\Omega)\times\R^3)
 \end{equation*}
 for any initial value $u_1=(v_1,\sa_1)\in B_{X_{\gamma,\mu}}(u_0,\eta)$. 
 Furthermore, there exists a positive constant $c=c(u_0)$ such that 
\begin{equation}
\label{eq:continuous_dep}
\|u(\cdot, u_1)-u(\cdot,u_2)\|_{\EE_{1,\mu}(0,T)}\le c|u_1-u_2|_{X_{\gamma,\mu}}
\end{equation}
for all $u_i=(v_i, \sa_i)\in B_{X_{\gamma,\mu}}(u_0,\eta)$, $i=1,2$. 
\item[{\bf(b)}]
Suppose $p_j, q_j$, $\mu_j$ satisfy \eqref{assumptions-pq}-\eqref{assumptions-mu}  and, in addition, $p_1\leq p_2$, $q_1\leq q_2$ as well as
\begin{equation}\label{mu-j}
 \mu_1- \frac{1}{p_1}- \frac{3}{2q_1} \ge  \mu_2- \frac{1}{p_2}- \frac{3}{2q_2}.
 \end{equation}
Then for each initial value
$(v_0,{\sf a}_0)\in {_\parallel}B^{2\mu_1 -2/p_1}_{q_1 p_1,\sigma}(\Omega)\times \R^3,$ 
problem  \eqref{eq:evolution} admits a unique solution $(v,{\sf a})$ in the class
\begin{equation*}
\begin{split}
&H^1_{p_1,\mu_1}((0,T); L_{q_1,\sigma}(\Omega)\times\R^3)\cap L_{p_1,\mu_1}((0,T); {_\parallel H}^2_{q_1,\sigma}(\Omega)\times\R^3) \\
&\cap H^1_{p_2,\mu_2}((0,T); L_{q_2,\sigma}(\Omega)\times\R^3)\cap L_{p_2,\mu_2}((0,T); {_\parallel H}^2_{q_2,\sigma}(\Omega)\times\R^3).
\end{split}
\end{equation*}
\item[{\bf (c)}]
Each solution with initial value in  $(v_0,\sa_0)\in {_\parallel}B^{2\mu-2/p}_{qp,\sigma}(\Omega)\times \R^3$
exists on a maximal interval  $[0,t_+)=[0,t_+(v_0,\sa_0))$, and 
enjoys the additional regularity property
\begin{equation*}
\qquad \qquad v \in C([0,t_+); {_\parallel}B^{2\mu-2/p}_{qp,\sigma}(\Omega))\cap C((0,t_+);{_\parallel}B^{2-2/p}_{qp,\sigma}(\Omega)),
\;\; \sa\in C^1([0,t_+),\R^3).
\end{equation*}
Hence, $v$ regularizes instantaneously if $\mu<1$.
\end{enumerate}
\end{theorem}
\begin{proof}
By Proposition~\ref{pro:L-calculus}, the operator $L$ in \eqref{eq:evolution} has the same properties as 
the corresponding operator $L$ in \cite{MaPrSi18} 
(which is generated by the Stokes operator with no-slip boundary conditions), while the nonlinearity
in \eqref{eq:evolution} is identical to the one in  \cite{MaPrSi18}.
Therefore, the proof of \cite[Theorem 3.4]{MaPrSi18} carries over to the current situation.
\end{proof} 
\begin{remarks}\label{remark-2}
{\rm {\bf (a)}
It has been shown in~\cite{PrSiWi18} that the concept of \textit{critical weight} $\mu_{\rm crit}$ and \textit{critical space} 
$X_{\rm {crit}}:=X_{\mu_{\rm crit}-1/p}$
 captures and unifies
the idea of `largest space for well-posedness,' and `scaling invariant space.' In more detail, it has been shown in~\cite{PrSiWi18} that
 $X_{\rm {crit}}$ is, in a generic sense, the largest space of initial data for which the given equation is $L_{p,\mu}$-well-posed, and 
that   $X_{\rm {crit}}$ is scaling invariant, provided the given equation has this property.

We note that the case of  $p=q=2$ is permissible in Theorem~\ref{th:local_strong} and yields
$\mu_{\rm crit}=3/4$. Hence
\[
X_{\rm {crit}}=(L_{2,\sigma}(\Omega), {_\parallel H}^2_{2,\sigma}(\Omega))_{1/4,2}= {_\parallel H}^{1/2}_{2,\sigma}(\Omega),
\]
 reminiscent of the celebrated Fujita-Kato Theorem \cite{FuKa62} for the Navier-Stokes~system.
\goodbreak
\noindent
{\bf (b)}
Let $\mu = 1/p+3/2q-1/4$ in Theorem \ref{th:local_strong}.
Then $1/p < \mu\le 1$ yields the restrictions
\begin{equation}
\label{mu-H1}
2/p+3/q\le5/2\quad \text{and}\quad q<6. 
\end{equation}
Hence, Theorem \ref{th:local_strong}(a) and (c) still holds under these assumptions. 
In order to see this, it suffices to show that there exists $\beta\in (\mu-1/p,1)$ such that the bilinear mapping
$$G:  H^{2\beta}_q(\Omega)^3\times H^{2\beta}_q(\Omega)^3 
\to L_q(\Omega)^3, \quad (v_1,v_2)\mapsto v_1\cdot \nabla v_2,$$
is continuous. By H\"older's inequality and Sobolev embedding we have
$$ |v_1\cdot \nabla v_2|_{L_q(\Omega)}\le |v_1|_{L_{qr}(\Omega)} |v_2|_{H^1_{qr^\prime}(\Omega)}
\le c |v_1|_{H^{2\beta}_q(\Omega)} |v_2|_{H^{2\beta}_q(\Omega)},$$
provided 
\begin{equation}
\label{eq:embeddingqrr}
3/q- 3/qr\le 2\beta, \quad1+3/qr\le 2\beta,
\end{equation}
respectively. 
Choose $r = 3$. Then for  any  $q$ satisfying $3/q<5/2$  there is  $\beta\in (\mu-1/p,1)$ such that \eqref{eq:embeddingqrr} holds.
Indeed, the restriction $3/q<5/2$, which is already covered by~\eqref{mu-H1}, ensures that $\mu-1/p<1$.
Since we can choose $\beta$ as close to 1 as we wish, we only need to require
\[
3/q- 3/qr< 2, \quad1+3/qr< 2,
\]
which holds for any $q>1$  in case $r=3$. 
Note that adding the two inequalities in \eqref{eq:embeddingqrr} yields $1 + 3/q \le 4\beta $. 
\smallskip\\
\noindent
{\bf (c)}
Theorem~\ref{th:local_strong}(b) and Remark (b) assert that problem \eqref{eq:evolution} admits for each initial value
$$(v_0,{\sf a_0})\in {H}^1_{2,\sigma}(\Omega)\times\R^3$$
a unique solution in the class 
\begin{equation*}
\begin{split}
&H^1_{2}((0,T); L_{2,\sigma}(\Omega)\times\R^3)\cap L_{2}((0,T); {_\parallel H}^2_{2,\sigma}(\Omega)\times\R^3) \\
&\cap H^1_{p,\mu}((0,T); L_{q,\sigma}(\Omega)\times\R^3)\cap L_{p,\mu}((0,T); {_\parallel H}^2_{q,\sigma}(\Omega)\times\R^3),
\end{split}
\end{equation*}
for any $p\ge 2$, $q\in [2,6)$, with $\mu=1/p +3/2q-1/4$.
In particular, we can conclude that $v\in C((0,t_+); B^{2-2/p}_{qp}(\Omega))$ for any $p\ge 2$, $q\in [2,6).$ 
In this class of solutions Lemma \ref{lem:energy_balance} holds.

Let us take the $L_2$-inner product of \eqref{problem}$_1$ with $\partial_t v$. Using \eqref{problem}$_{2,3,4}$ and 
Lemma~\ref{lem:positive_def} together with Young's inequality, we find 
\[\begin{split}
\frac 12 \frac{d}{dt}\left[2\upnu |D(v)|^2_{L_2(\Omega)}+\zeta|v|^2_{L_2(\Gamma)}\right]&+\frac c2|\partial_t v|^2_{L_2(\Omega)}
\\
\le c_1&\left[|\dot\sa|^2+(|\sa|^2+|\omega|^2)|v|^2_{L_2(\Omega)}+|v\cdot\nabla v|^2_{L_2(\Omega)}\right].
\end{split}\]
Now consider the Helmholtz projection of \eqref{problem}$_1$. Using H\"older and Young inequalities, we obtain 
\[
\upnu|\PP \Delta v|^2_{L_2(\Omega)}\le c_2\left[ |\partial_t v|^2_{L_2(\Omega)}+|\dot\sa|^2+
(|\sa|^2+|\omega|^2)|v|^2_{L_2(\Omega)}+|v\cdot\nabla v|^2_{L_2(\Omega)}\right].
\]
From the latter two displayed inequalities  and 
\[
|v|_{H^2_2(\Omega)}\le C_1|\PP \Delta v|_{L_2(\Omega)} ,\]
which is a consequence of  Proposition~\ref{pro:L-calculus}(a), it follows that  
\[\begin{split}
\frac 12 \frac{d}{dt}\left[2\upnu |D(v)|^2_{L_2(\Omega)}+\zeta|v|^2_{L_2(\Gamma)}\right]&+\frac c4 |\partial_t v|^2_{L_2(\Omega)}
+c_3|v|^2_{H^2_2(\Omega)}
\\
\le c_4&\left[|\dot\sa|^2+(|\sa|^2+|\omega|^2+1)|v|^2_{L_2(\Omega)}+|v\cdot\nabla v|^2_{L_2(\Omega)}\right].
\end{split}\]
Let us estimate the right-hand side of the above inequality. By \eqref{problem}$_5$ and Lemma \ref{lem:energy_balance}, we get 
\[
|\dot\sa|^2+(|\sa|^2+|\omega|^2+1)|v|^2_{L_2(\Omega)}\le C_3
\]
where $C_3$ is a positive constant depending only on initial data and physical and geometric properties of $\mathcal S.$ Moreover, by H\"older's inequality, Sobolev embedding and the interpolation inequality together with Young's inequality, we find 
\[\begin{split}
|v\cdot\nabla v|^2_{L_2(\Omega)}&\le |v|^2_{L_6(\Omega)}|\nabla v|^2_{L_3(\Omega)}\le C_4 |v|^2_{H^1_2(\Omega)}|\nabla v|^2_{L_3(\Omega)}
\\
&\le C_4|v|^2_{H^1_2(\Omega)}|\nabla v|_{L_2(\Omega)}|\nabla v|_{L_6(\Omega)}
\le C_5|v|^3_{H^1_2(\Omega)}|v|_{H^2_2(\Omega)}
\\
&\le C_6 |v|^6_{H^1_2(\Omega)}+\frac{c_3}{2c_4}|v|^2_{H^2_2(\Omega)}. 
\end{split}\]
Using \eqref{eq:korn}, we conclude that 
\begin{equation}\label{eq:grad_estimate}\begin{split}
\frac 12 \frac{d}{dt}\left[2\upnu |D(v)|^2_{L_2(\Omega)}+\zeta|v|^2_{L_2(\Gamma)}\right]&+\frac c4 |\partial_t v|^2_{L_2(\Omega)}
+\frac{c_3}{2}|v|^2_{H^2_2(\Omega)}
\\
&\le c_5[8\upnu^3 |D(v)|^6_{L_2(\Omega)}+\zeta^3|v|^6_{L_2(\Gamma)}+1].
\end{split}\end{equation}
This differential inequality implies the following blow-up criterion: either $t_+=\infty$ or else, if $t_+<\infty$ then 
\[
\lim_{t\to t_+^-}\left[2\upnu |D(v)|^2_{L_2(\Omega)}+\zeta|v|^2_{L_2(\Gamma)}\right]=\infty. 
\]}
\end{remarks}  

%
\section{Existence of global weak solutions and related properties}\label{sec:weak}
In this section we will show that the class of weak solutions {\em \`a la Leray-Hopf} of \eqref{problem} is nonempty. Such solutions are global in time for data having finite initial kinetic energy, and they possess the further property of becoming ``regular'' (and unique) after a sufficiently large time. After this time, the equations of motion are satisfied a.e. in space and time.  

Let us recall the following Friedrichs and Korn type inequalities. 
\begin{lemma}\label{lem:FK}
There exists a constant $0<K_1<1$ such that
\begin{equation}\label{eq:korn}
K_1|u|_{H^1_q(\Omega)}\le\left(|u|_{L_q(\Gamma)}^q+|D(u)|^q_{L_q(\Omega)}\right)^{1/q}
\end{equation}
for all $u\in H^1_q(\Omega)^3$. 

Moreover, there exists a positive constant $K_2$ such that, for every $v\in H^1_2(\Omega)^3$ such that $(v|n)=0$ and 
$2\upnu P_\Gamma(D(v)n)+\zeta v=0$ on $\Gamma$, 
\begin{equation}\label{eq:friedrichs-korn}
|v|_{H^1_2(\Omega)}\le K_2 |D(v)|_{L_2(\Omega)}. 
\end{equation}
\end{lemma}
\begin{proof} 
The proof of \eqref{eq:korn} can be found in \cite[Equation (26)]{SoSc73}.  
The proof of \eqref{eq:friedrichs-korn} follows from two important inequalities. Friedrichs' inequality (\cite[Exercise II.5.15]{Ga})
\[
|v|_{L_2(\Omega)}\le c_1 |\nabla v|_{L_2(\Omega)}
\]
holds for every $v\in H^1_{2}(\Omega)^3$ such that $(v|n)=0$ on $\Gamma$. Korn's inequality 
\[
|\nabla v|_{L_2(\Omega)}\le c_2|D(v)|_{L_2(\Omega)}
\] 
is satisfied for every $v\in H^1_{2}(\Omega)^3$ such that $(v|n)=0$ and $2\upnu P_\Gamma(D(v)n)+\zeta v=0$ on $\Gamma$, see \cite[Lemma 4]{SoSc73}. We wish to emphasize that, as remarked by the authors in \cite[Remark 2]{SoSc73}, this latter inequality holds for every $v\in H^1_{2}(\Omega)^3$ such that $(v|n)=0$ on $\Gamma$ if $\Omega$ is not a solid of revolution around a vector $\omega_0$, otherwise it holds for every $v$ belonging to any subspace of $H^1_2(\Omega)^3$ not containing $\omega_0$ and such that $(v|n)=0$ on $\Gamma$. In our case, the boundary condition $2\upnu P_\Gamma(D(v)n)+\zeta v=0$ on $\Gamma$ ensures the validity of \eqref{eq:friedrichs-korn} also in the case $\Omega$ is  a solid of revolution around a vector $\omega_0$. 
\end{proof}

Consider the bilinear  form \[
b(u,v):=\zeta(u|v)_\Gamma+2\upnu (D(u)|D(v))_\Omega
\quad \text{for every  $u,v\in H^1_2(\Omega)^3$. }
\]
In particular, for every $v\in H^1_2(\Omega)^3$ 
\[
b(v,v)=\zeta|v|^2_{L_2(\Gamma)}+2\upnu |D(v)|^2_{L_2(\Omega)}.
\]
Hence, $b(v,v)=0$ iff $v=0$ on $\Gamma$ and $D(v)=0$ in $\Omega$. We can then infer that $b(\cdot,\cdot)$ with domain  $H^1_2(\Omega)^3\times H^1_2(\Omega)^3$ is a positive definite bilinear form. 
A simple integration by parts shows that $b(u,v)=(A_N u|v)_\Omega$ for every $u\in {_\parallel}H^2_{2,\sigma}(\Omega)$ and $v~\in~H^1_{2,\sigma}(\Omega).$ Using Lemma \ref{lem:FK}, we conclude that the operator $A_N$ with domain $D_2(A_N)\equiv {_\parallel}H^2_{2,\sigma}(\Omega)$ is invertible, self-adjoint and positive definite. Moreover, since 
$H^2_2(\Omega)^3$ is compactly embedded in $L_2(\Omega)^3$, then $A_N^{-1}$ is compact. Thus, $\sigma(A_N)\subset(0,\infty)$ and consists solely of eigenvalues $\{\Lambda_n\}_{n\in \N}$ clustering at $+\infty$. The standard spectral theory and elliptic regularity imply the existence of an orthonormal basis of $L_{2,\sigma}(\Omega)$ of eigenfunctions $\{\varphi_n\}_{n\in \N}\subset D_2(A_N)$ for $A_N$. Furthermore, the bilinear form $b(\cdot,\cdot)$ defines the inner product 
\begin{equation}\label{eq:inneraH12}
\langle u,v\rangle_{H^1_2(\Omega)}:=b(u,v)=\zeta (u|v)_\Gamma+2\upnu (D(u)|D(v))_\Omega,\qquad u,v\in H^1_{2,\sigma}(\Omega),
\end{equation}
with associated norm equivalent to $|\cdot|_{H^1_2(\Omega)}$, thanks to Lemma \ref{lem:FK} and trace theory. Then, $\{\varphi_n/\sqrt \Lambda_n\}_{n\in \N}$ is orthonormal in $H^1_{2,\sigma}(\Omega)$ endowed with the inner product \eqref{eq:inneraH12}. In fact, if $u\in H^1_{2,\sigma}(\Omega)$ satisfies $\langle u,\varphi_n/\sqrt\Lambda_n\rangle_{H^1_2(\Omega)}=0$ for every $n\in \N$, then
\[
0=\frac{1}{\sqrt{\Lambda_n}}b(\varphi_n,u)=\frac{1}{\sqrt{\Lambda_n}}(A_n\varphi_n|u)_\Omega=\sqrt{\Lambda_n}(\varphi_n|u)_\Omega\;\Rightarrow\, u=0,
\]
showing that $\{\varphi_n/\sqrt \Lambda_n\}_{n\in \N}$ is complete. 
Finally, 
\begin{multline*}
\langle \frac{\varphi_n}{\sqrt{\Lambda_n}},\frac{\varphi_m}{\sqrt{\Lambda_m}}\rangle_{H^1_2(\Omega)}=\frac{1}{\sqrt{\Lambda_n}\,\sqrt{\Lambda_m}}b(\varphi_n,\varphi_m)=\frac{1}{\sqrt{\Lambda_n}\,\sqrt{\Lambda_m}}(A_N\varphi_n|\varphi_m)_\Omega
\\
=\frac{\Lambda_n}{\sqrt{\Lambda_n}\,\sqrt{\Lambda_m}}\delta_{nm},\quad n,m\in \N. 
\end{multline*}

We summarize these properties in the following theorem (see also \cite[Theorem 4.11]{ChQi10} 
where an analogous result is proved for the Navier-Stokes equations, 
using a reformulation of the Navier boundary condition \eqref{problem}$_4$ in terms of the vorticity). 

\begin{theorem}\label{th:basis}
The spectrum of the Stokes operator $A_N$ with Navier boundary condition is discrete and is contained in $(0,\infty)$. The eigenvalues $\{\lambda_n\}_{n\in \N}$ satisfy 
\[
0<\Lambda_0\le \Lambda_1\le\dots\le \Lambda_n\le \dots,\quad \lim_{n\to \infty}\Lambda_n=+\infty. 
\]
The corresponding eigenfunctions $\{\varphi_n\}_{n\in \N}\subset {_\parallel}H^2_{2,\sigma}(\Omega)$ form an orthonormal basis of $L_{2,\sigma}(\Omega).$ 
Moreover,  $\{\varphi_n/\sqrt \Lambda_n\}_{n\in \N}$ is an orthonormal basis of $H^1_{2,\sigma}(\Omega)$. 
\end{theorem}

This basis of eigenfunctions for the Stokes operator with Navier boundary conditions will be used to approximate solutions to \eqref{problem}$_{1,2}$ in the Leray-Hopf class. The weak formulation of \eqref{problem} can be obtained by dot-multiplying \eqref{problem}$_1$ by a test function $\phi \in H^1_{2,\sigma}(\Omega)$ and integrating the resulting equation first over space, and then in time. This leads to the following system of equations (recall also \eqref{finite-rank}):
\begin{equation}\label{eq:weak}
\begin{aligned}
&((I+K) v(t)|\phi)_\Omega +(\sa(t)|x\times \phi)_\Omega+ \int^t_0[(v\cdot\nabla v+2(\sa-\omega)\times v|\phi)_\Omega+b(v,\phi)]\; d\tau 
\\
&\qquad\qquad=((I+K) v(0)|\phi)_\Omega +(\sa(0)|x\times \phi)_\Omega,\text{ for all }\phi \in H^1_{2,\sigma}(\Omega),\; t\in (0,\infty),
\\
&\;\II\, \sa(t) + \int^t_0(\sa-\omega)\times\II\sa\; d\tau =\II\, \sa(0),\quad \text{for all }t\in (0,\infty). 
\end{aligned}
\end{equation}

\begin{defi}
The couple $(v,\sa)$ is a {\em weak solution} \`a la Leray-Hopf of \eqref{problem} if the following conditions are satisfied. 
\begin{enumerate}
\setlength\itemsep{1mm}
\item $v\in C_w ([0,\infty);L_{2,\sigma}(\Omega))\cap L_\infty((0,\infty);L_{2,\sigma}(\Omega))
\cap L_{2,loc}([0,\infty);H^1_{2,\sigma}(\Omega))$. 
\item $\sa\in C^0([0,\infty))\cap C^1((0,\infty))$.  
\item $(v,\sa)$ satisfies \eqref{eq:weak}. 
\item The {\em strong energy inequality} holds: 
\begin{equation}\label{eq:strong_energy}
{\sf E}(v(t),\sa(t))+2\upnu \int^t_s|D(v(\tau))|_{L_2(\Omega)}^2\; d\tau+\zeta\int^t_s|v(\tau)|_{L_2(\Gamma)}^2\;d\tau\le{\sf E}(v(s),\sa(s)),
\end{equation}
for all $t\ge s$ and a.a. $s\ge 0$ including $s=0$. 
\end{enumerate}
\end{defi}
The class of the above solutions is nonempty for initial data having finite kinetic energy. 
\begin{theorem}\label{th:weak}
For any initial value $(v_0, {\sf a}_0)\in L_{2,\sigma}(\Omega)\times \R^3$, 
there exists at least one weak solution $(v,\sa)$ \`a la Leray-Hopf such that 
\[
\lim_{t\to 0^+}|v(t)-v_0|_{L_2(\Omega)}=\lim_{t\to 0^+}|\sa(t)-\sa_0|=0. 
\]
\end{theorem}

\begin{proof}
The existence of a global weak solution will be accomplished using the Galerkin method with the basis constructed in Theorem \ref{th:basis}. Consider the approximating solutions 
\[
v_n(t,x)=\sum^n_{k=0}c_{nk}(t)\varphi_k(x),\quad \sa_n(t)=\sum^3_{i=1}\sa_{ni}(t)\, {\sf e}_i,\qquad n\in \N,\] 
satisfying 
\begin{equation}\label{eq:galerkin0}\begin{aligned}
&((I+K) v_n(t)|\varphi_r)_\Omega +(\sa_n(t)|x\times \varphi_r)_\Omega+ \int^t_0(v_n\cdot\nabla v_n+2((\sa_n-\omega_n)\times v_n|\varphi_r)_\Omega\; d\tau 
\\
&\qquad+\int^t_0b(v_n,\varphi_r)\;d\tau=((I+K) v_n(0)|\varphi_r)_\Omega +(\sa_n(0)|x\times \varphi_r)_\Omega,\quad   r=1,\dots n, 
\\
&\;\II\, \sa_n(t) + \int^t_0(\sa_n-\omega_n)\times\II\sa_n\; d\tau =\II\, \sa_n(0), 
\\
&v_n(0)=\sum^n_{k=0}(v_0|\varphi_k)_\Omega\, \varphi_k,\qquad \sa_n(0)=\sum^3_{i=1}(\sa_0|{\sf e}_i)\, {\sf e}_i,
\end{aligned}\end{equation}
where 
\[
\omega_n=\sum^n_{k=0}c_{nk}(t)\,\II^{-1}\int_\Omega x \times \varphi_k\,dx. 
\]
The coefficients $c_{nr}$ and $a_{ni}$ have to satisfy the following system of differential equations (with summation on repeated indices ranging from $0$ to $n$ and obviously, without summation on $n$ and $i$): 
\begin{equation}\label{eq:galerkin}
\begin{aligned}
&B^{(n)}_{rk}\dot c_{nk}+C^{(n)}_{rk}c_{nk}+D^{(n)}_{rk\ell}c_{nk}c_{n\ell}+E^{(n)}_r=0,\quad c_{nk}(0)=(v_0|\varphi_k)_\Omega,\quad\  r=1,\dots n, 
\\
&\;\lambda_i\dot\sa_{ni} +( (\sa_n-\omega_n)\times\II\sa_n|{\sf e}_i) =0, \quad \sa_{ni}(0)=(\sa_0|{\sf e}_i),\quad i=1,2,3,
\end{aligned}\end{equation}
where
\[\begin{split}
&B^{(n)}_{rk}:=\delta_{rk}+(K\varphi_k|\varphi_r)_\Omega,
\\
&C^{(n)}_{rk}:=\left[\left((\II^{-1}\int_\Omega x\times \varphi_k\;dx)\times \II\sa_n\left|\right.\II^{-1}\int_\Omega x\times \varphi_r\;dx\right)\right.
\\
&\qquad\qquad\qquad\qquad\qquad\qquad\qquad\qquad\qquad\left.+2(\sa_n|\int_\Omega \varphi_k\times\varphi_r\;dx)+b(\varphi_k,\varphi_r)\right],
\\
&D^{(n)}_{rk\ell}:=(\varphi_k\cdot\nabla \varphi_\ell|\varphi_r)_\Omega-2\left(\int_{\Omega}\varphi_k\times \varphi_r\,dx\right|\left.\II^{-1}\int_\Omega x \times \varphi_\ell\,dx\right),
\\
&E^{(n)}_r:=-\left(\sa_n\times \II\sa_n\left|\right.\II^{-1}\int_\Omega x\times\varphi_r\;dx\right). 
\end{split}\]
For every $n\in \N$, \eqref{eq:galerkin} is a system of first order, quadratic ordinary differential equations with constant coefficients, and it admits a unique solution defined in some interval $[0,T_n)$ with $T_n > 0$ (note that the operator $I+K$ is invertible, as shown in the proof of Proposition~\ref{pro:E-invertible}). Actually, $T_n=+\infty$ since the following holds
\[
\frac 12 \frac{d }{d t}\left[((I+K)v_n|v_n)_\Omega+(\II\sa_n|\sa_n)\right]+2\upnu |D(v_n)|^2_{L_2(\Omega)}+\zeta|v_n|^2_{L_2(\Gamma)}=0,
\]
and by \eqref{I+K positive}, \eqref{energy} and Lemma \ref{lem:positive_def}, it implies the uniform energy estimate 
\begin{multline}\label{eq:energy_n}
{\sf E}(v_n(t),\sa_n(t))+2\int^t_0\upnu |D(v_n(\tau))|^2_{L_2(\Omega)}\;d\tau+\zeta\int^t_0|v_n(\tau)|^2_{L_2(\Gamma)}\;d\tau={\sf E}(v_n(0),\sa_n(0))
\\
\le |v_0|^2_{L_2(\Omega)}+(\II\sa_0|\sa_0). 
\end{multline}
The latter inequality provides also the following important information. 
\begin{enumerate}
\item[{\sf (a)}] $\{v_n\}_{n\in \N}$ is uniformly bounded in $L_\infty((0,\infty); L_{2,\sigma}(\Omega)).$
\item[{\sf (b)}] $\{D(v_n)\}_{n\in \N}$ is uniformly bounded in $L_2((0,\infty);L_{2,\sigma}(\Omega)^{3\times 3})$ 
and $\{v_n\}_{n\in \N}$ is uniformly bounded in $L_2((0,\infty); L_2(\Gamma)^3).$ By Lemma \ref{lem:FK}, the sequence $\{v_n\}_{n\in \N}$ is uniformly bounded in $L_{2,loc}([0,\infty);H^1_{2,\sigma}(\Omega)).$
\item[{\sf (c)}] $\{\sa_n\}_{n\in \N}$ is uniformly bounded in $BC^1([0,\infty)).$
\end{enumerate}
For every $n\in \N$, we denote by $\mathcal P_n$ the orthogonal projecton of $H^1_{2,\sigma}(\Omega)$ onto the linear span of $\{\varphi_1/\sqrt{\Lambda_1},\dots,\varphi_n/\sqrt{\Lambda_n}\}$. By Theorem \ref{th:basis}, we infer that 
\[
|\mathcal P_n w|_{H^1_2(\Omega)}\le |w|_{H^1_2(\Omega)}\quad\text{and}\quad \lim_{n\to \infty}\mathcal P_nw=w\quad\text{in }H^1_{2,\sigma}(\Omega). 
\]
From \eqref{eq:galerkin0}, for every $w\in H^1_{2,\sigma}(\Omega)$ one has 
\begin{multline*}
((I+K) \dot v_n|w)_\Omega=((I+K) \dot v_n|\mathcal P_n w)_\Omega
\\
\quad=-(\dot\sa_n|x\times \mathcal P_n w)_\Omega-(v_n\cdot\nabla v_n+2((\sa_n-\omega_n)\times v_n|\mathcal P_n w)_\Omega-b(v_n,\mathcal P_n w).
\end{multline*}
Using H\"older's inequality as well as interpolation inequalities, the trace theorem and Lemma \ref{lem:FK}, the right-hand side of the latter displayed equation can be estimated as follows with a positive constant $c$ independent of $n$
\[\begin{split}
|((I+K) \dot v_n|w)_\Omega|&\le c \left[(|\sa_n|+|D(v_n)|_{L_2(\Omega)})|\sa_n|+|v_n|^{1/2}_{L_2(\Omega)}|D(v_n)|^{3/2}_{L_2(\Omega)}\right.
\\
&\qquad\qquad\qquad\qquad\qquad \left.+(|v_n|_{L_2(\Omega)}+1)|D(v_n)|_{L_2(\Omega)}\right]|w|_{H^1_2(\Omega)}. 
\end{split}\]
As a consequence, we have 
\begin{enumerate}
\item[{\sf (d)}] $\{(I+K) \dot v_n\}_{n\in \N}$ remains in a bounded set of $L_{4/3,loc}([0,\infty); (H^1_{2,\sigma}(\Omega))')$. By \eqref{I+K positive} and (c), we also infer that  $\{(I+K) v_n\}_{n\in \N}$ is uniformly bounded in $L_{2,loc}([0,\infty);H^1_{2,\sigma}(\Omega))$. 
\end{enumerate}
Properties {\sf (a)}--{\sf (c)} imply the existence of functions
\[
v\in L_{2,loc}([0,\infty);H^1_{2,\sigma}(\Omega))\cap L_\infty((0,\infty);L_{2,\sigma}(\Omega)),\quad \sa\in C([0,\infty))
\] 
and subsequences $\{(v_{n_k},\sa_{n_k})\}_{n_k\in \N}$ such that, for every $T>0$,
\begin{equation}\label{eq:convergence1}
\begin{split}
&v_{n_k}\rightharpoonup v\quad \text{weakly in }L_{2}((0,T);H^1_{2,\sigma}(\Omega))
\\
&v_{n_k}\rightharpoonup v\quad \text{weakly-star in }L_\infty((0,\infty);L_{2,\sigma}(\Omega)),
\\
&\sa_{n_k}\to \sa\quad \text{uniformly in }[0,T]. 
\end{split}
\end{equation}
Since $H^1_{2,\sigma}(\Omega)\hookrightarrow L_{2,\sigma}(\Omega)\hookrightarrow (H^1_{2,\sigma}(\Omega))'$ with the first embedding being compact, by the Aubin-Lions Lemma (see \cite[Theorem 2.1]{Te01}) and \eqref{I+K positive}, we conclude that 
\begin{equation}\label{eq:convergence2}
v_n\to v\quad \text{strongly in }L_{2}((0,T);L_{2,\sigma}(\Omega)). 
\end{equation}
The convergence results \eqref{eq:convergence1} and \eqref{eq:convergence2} allow to pass to limit in \eqref{eq:galerkin0} and \eqref{eq:energy_n}. We omit this proof as it is standard (see e.g. \cite{Te01,Ma12}). 

Note that as a byproduct of \eqref{eq:weak}$_2$, we also have that $\sa\in C^1((0,\infty))$ and $v\in C_w([0,\infty);(H^1_{2,\sigma}(\Omega))')$. Since $v\in L_\infty((0,\infty);L_{2,\sigma}(\Omega))\cap C_w([0,\infty);(H^1_{2,\sigma}(\Omega))')$, by \cite[Lemma 1.4]{Te01}, we have that $v\in C_w([0,\infty);L_{2,\sigma}(\Omega))$. Finally, the map 
\[
t\mapsto  |v(t)|_{L_2(\Omega)}
\]
is lower semicontinuous at $t = 0$. By Lemma \ref{lem:positive_def} and the strong energy inequality with $s = 0$, we then infer that \[
\lim_{t\to 0^+}|v(t)-v_0|_{L_2(\Omega)}=\lim_{t\to 0^+}|\sa(t)-\sa_0|=0. 
\]
\end{proof}

\begin{remark}\label{rem:regularity_w}
If $v\in L_\infty((0,T);L_{2,\sigma}(\Omega))\cap L_2([0,T);H^1_{2,\sigma}(\Omega))$ then 
\[
v\in L_p([0,T); L_q(\Omega)^3),\qquad \text{where }\quad\frac 2p+\frac 3q=\frac 32,\ p\in [2,\infty]. 
\]
In fact, by interpolation, for $3/q+2/p=3/2$ 
\[
|v|_{L_q(\Omega)}\le |v|_{H^1_2(\Omega)}^{2/p}|v|_{L_2(\Omega)}^{1-2/p}.
\]
Then, 
\begin{multline*}
|v|_{L_p([0,T);L_q(\Omega))}\le\left(\int^T_0|v|_{H^1_2(\Omega)}^{2}|v|_{L_2(\Omega)}^{(1-2/p)p}d\tau\right)^{1/p}
\\
\le |v|_{L_\infty((0,T);L_2(\Omega))}^{1-2/p}|v|_{L_2([0,T);H^1_2(\Omega))}^{2/p}. 
\end{multline*}
\end{remark}

As for the classical Navier-Stokes equations, also for the problem at hand, it is still an open question whether 
weak solutions are unique.  However, we can show that such property holds if weak solutions are ``slightly'' more regular. 

\begin{theorem}\label{th:cduid}
Let $(v_1,\sa_1)$ and $(v_2,\sa_2)$ be two weak solutions corresponding to initial data $(v_{0},\sa_{0})$. 
Suppose that for some $T>0$
\[
v_2\in L_s([0,T);L_r(\Omega)^3),\ \text{with }s\in(2,\infty),\, r\in (3,\infty)\text{ satisfying }\frac 2s+\frac 3r=1. 
\]
Then, $(v_1,\sa_1)\equiv (v_2,\sa_2)$ in $[0,T).$   
\end{theorem}

\begin{proof}
The proof of this theorem is rather standard (see e.g. \cite[1.5.1 Theorem]{So01}, \cite[Theorem 3.4.2]{Ma12}). 
Here we will only provide the main estimates.  

By the classical mollification method (see \cite[II Section 1.7]{So01}), one can show that 
\begin{equation}\label{eq:v1v2}
\begin{aligned}
&((I+K) v_1(t)|v_2(t))_\Omega+(\II\sa_1(t)|\sa_2(t))=((I+K) v_0|v_0)_\Omega+(\II\sa_0|\sa_0)
\\
&\quad\quad-\int^t_0[(v_1\cdot\nabla v_1+2(\sa_1-\omega_1)\times v_1|v_2)_\Omega+(v_2\cdot\nabla v_2+2(\sa_2-\omega_2)\times v_2|v_1)_\Omega]\; d\tau
\\
&\quad \quad+\int^t_0[((\sa_1-\omega_1)\times \II \sa_1|\omega_2-\sa_2)+((\sa_2-\omega_2)\times \II \sa_2|\omega_1-\sa_1)-2b(v_1,v_2)]d\tau, 
\end{aligned}
\end{equation}
for all $t\in [0,T)$, where 
\[
\omega_1=\II^{-1}\int_\Omega x \times v_1\,dx,\qquad \omega_2=\II^{-1}\int_\Omega x \times v_2\,dx. 
\]
We recall that both $(v_1,\sa_1)$ and $(v_2,\sa_2)$ satisfy the strong energy inequality:  
\[\begin{split}
&{\sf E}(v_1(t),\sa_1(t))+2\upnu \int^t_0|D(v_1(\tau))|_{L_2(\Omega)}^2\; d\tau+\zeta\int^t_0|v_1(\tau)|_{L_2(\Gamma)}^2\;d\tau\le{\sf E}(v_0,\sa_0),
\\
&{\sf E}(v_2(t),\sa_2(t))+2\upnu \int^t_0|D(v_2(\tau))|_{L_2(\Omega)}^2\; d\tau+\zeta\int^t_0|v_2(\tau)|_{L_2(\Gamma)}^2\;d\tau\le{\sf E}(v_0,\sa_0). 
\end{split}\]
Recall \eqref{I+K positive} and \eqref{energy}. Adding the last two inequalities, and subtracting twice of \eqref{eq:v1v2}, we find that the fields $v := v_1-v_2$, $\omega:=\omega_1-\omega_2$ and $\sa:= \sa_1-\sa_2$ must satisfy the following inequality
\begin{equation}\label{eq:energy_uniqueness}\begin{split}
{\sf E}(v(t),\sa(t))&+2\upnu \int^t_0|D(v(\tau))|_{L_2(\Omega)}^2\; d\tau+\zeta\int^t_0|v(\tau)|_{L_2(\Gamma)}^2\;d\tau
\le
\\
&\qquad \int^t_0[(v\cdot\nabla v+2(\sa-\omega)\times v|v_2)_\Omega+((\sa-\omega)\times \II \sa|\sa_2-\omega_2)]\;d\tau.
\end{split}\end{equation}
Take $p>2$ and $q\in (2,6)$ such that $1/q+1/r=1/2$ and $1/p+1/s=1/2$. Use Remark \ref{rem:regularity_w},  interpolation inequality with $\alpha=2/p=3/r$, Young's inequality and \eqref{eq:friedrichs-korn}, to obtain the following estimate 
\[\begin{split}
\int^t_0|(v\cdot\nabla v|v_2)_\Omega|\;d\tau&\le \int^t_0|v|_{L_q(\Omega)}|\nabla v|_{L_2(\Omega)}|v_2|_{L_r(\Omega)}\;d\tau
\\
&\le c_1\int^t_0|v|^{1-\alpha}_{L_2(\Omega)}|v|^{1+\alpha}_{H^1_2(\Omega)}|v_2|_{L_r(\Omega)}\;d\tau
\\
&\le m K_1^2\int^t_0|v|^{2}_{H^1_2(\Omega)}\; d\tau
+c_2\int^t_0|v|^{2}_{L_2(\Omega)}|v_2|^s_{L_r(\Omega)}\;d\tau
\\
&\le m \int^t_0[|D(v)|^{2}_{L_2(\Omega)}+|v|^2_{L_2(\Gamma)}]\; d\tau
+c_2\int^t_0|v|^{2}_{L_2(\Omega)}|v_2|^s_{L_r(\Omega)}\;d\tau,
\end{split}\]
where $2m:=\min\{2\upnu,\zeta\}>0$. 
Estimating the other nonlinear terms in a similar (actually less sophisticated) way and using Lemma \ref{lem:positive_def}, we obtain 
\[\begin{split}
{\sf E}(v(t),\sa(t))&+m \int^t_0|D(v(\tau))|_{L_2(\Omega)}^2\; d\tau+m\int^t_0|v(\tau)|_{L_2(\Gamma)}^2\;d\tau
\\
&\qquad \le
c_3\int^t_0[|v_2(\tau)|^s_{L_r(\Omega)}+ |v_2(\tau)|_{L_2(\Omega)}+|\sa_2(\tau)|]\,{\sf E}(v(\tau),\sa(\tau))\;d\tau.
\end{split}\]
By Gronwall's Lemma we conclude that ${\sf E}(v(t),\sa(t))\equiv 0$ for every $t\in [0,T)$. By \eqref{energy} and Lemma \ref{lem:positive_def}, the uniqueness property immediately follows.  
\end{proof}
The following properties of weak solutions to \eqref{problem} will be fundamental in ascertaining the long-time behavior in 
Section~\ref{se:long-time}. 
\begin{theorem}\label{th:weak-strong}
Let $(v,\sa)$ be a weak solution to \eqref{problem} corresponding to initial data $(v_0, {\sf a}_0)\in L_{2,\sigma}(\Omega)\times \R^3$, as from Theorem \ref{th:weak}. Then, 
there exists a time $\tau>0$  such that  
\begin{equation*}
\begin{split}
v\in H^1_{p,\,loc}([\tau,\infty);L_{q,\sigma}(\Omega))
\cap L_{p,\,loc}([\tau,\infty);{_\parallel H}^2_{q,\sigma}(\Omega)),\quad \sa\in C^1([\tau,\infty);\R^3),
\end{split}\end{equation*} 
for each fixed $p\in [2,\infty)$, $q\in (1,6)$.
Moreover, there exists a corresponding pressure field $p \in L_{p,\, loc}([\tau,\infty); H^1_q(\Omega))$, such that
$(v, p, \sa)$ satisfy \eqref{problem}$_{1,2}$, almost everywhere in $\Omega\times(0,\infty).$ 
 Finally, $v\in BC([\tau,\infty);B^{2-2/p}_{qp}(\Omega)^3)$, and 
\begin{equation}\label{eq:bvto0}
\lim_{t\to \infty}|v(t)|_{H^{2\alpha}_q(\Omega)}=0,\quad \text{for any fixed }\alpha\in[0,1).
\end{equation}
\end{theorem}
\begin{proof}
Let $S:=\{s\in [0,\infty):\text{the strong energy inequality \eqref{eq:strong_energy} holds}\}$. 
By definition and \eqref{eq:korn}, $[0,\infty)\setminus S$ has zero Lebesgue measure, and for every $\delta$ and $\eta$ 
there exists $t_0\in S$, such that 
\begin{equation}\label{eq:eta_delta}
\begin{split}
2\upnu\int^\infty_{t_0}|D(v(\tau))|_{L_2(\Omega)}^2\; d\tau+\zeta\int^\infty_{t_0}|v(\tau)|_{L_2(\Gamma)}^2\;d\tau<\eta
\\
2\upnu|D(v(t_0))|_{L_2(\Omega)}^2+\zeta|v(t_0)|_{L_2(\Gamma)}^2<\delta. 
\end{split}\end{equation}
For $p\ge 2$ and $q\in[2,6)$, let
$\mu ={1}/{p}+ {3}/{2q}- {1}/{4},$ yielding
$H^1_2(\Omega)\hookrightarrow B^{2\mu-2/p}_{qp}(\Omega)$.
Next, we take $(v(t_0), {\sa}(t_0))\in {H}^1_{2,\sigma}(\Omega)\times \R^3$ as initial condition for 
a strong solution $(\tilde v,\tilde \sa)$ to \eqref{problem} on $[t_0, t_1]$, $t_1 \in (t_0, t_+)$, in
the class 
\[\begin{split}
&H^1_{p,\mu}((t_0,t_1);L_{q,\sigma}(\Omega)\times \R^3)\cap L_{p,\mu}((t_0,t_1);{_\parallel H}^2_{q,\sigma}(\Omega)\times \R^3)
\\
&\cap H^1_{2}((t_0,t_1);L_{2,\sigma}(\Omega)\times \R^3)\cap L_{2}((t_0,t_1);{_\parallel H}^2_{2,\sigma}(\Omega)\times \R^3), 
\end{split}\]
as from Theorem~\ref{th:local_strong}(b) and Remark \ref{remark-2}(b). 
Here we have used the notation
$$u\in L_{p,\mu}((t_0,T);X) \Leftrightarrow (t-t_0)^{1-\mu}u\in L_p((t_0,T);X)$$
and $u\in H^1_{p,\mu}((t_0,T);X) \Leftrightarrow u,\dot u\in L_{p,\mu}((t_0,T);X)$
for $0\le t_0<T\le \infty$ and $X$ a Banach space.

Since uniqueness of solutions holds in the above class, necessarily $v\equiv \tilde v$ and $\sa\equiv \tilde \sa$ on $[t_0, t_1]$, 
$t_1 \in (t_1, t_+)$, and $(v,\sa)$ satisfy \eqref{problem} a.e. on $[t_0, t_1]$, $t_1 \in (t_0, t_+).$   

In addition, $v$ satisfies \eqref{eq:grad_estimate}, implying the following differential inequality 
\[
\dot y\le -ay+by^3+c \quad \text{in }(t_0,t_+) 
\]
for the function $y(t):=2\upnu |D(v)|^2_{L_2(\Omega)}+\zeta|v|^2_{L_2(\Gamma)}$, and with positive constants $a,b,c.$ Choosing suitable $\eta$ and $\delta$ in \eqref{eq:eta_delta}, Gronwall's Lemma in \cite[Lemma 2.3.5]{Ma16} and Remark \ref{remark-2}(c) imply that $y(t)<c_0\delta$ for all $t\in [t_0,t_+)$ and then $t_+=\infty$. Furthermore, 
\[
\lim_{t\to \infty}|D(v(t))|^2_{L_2(\Omega)}=\lim_{t\to \infty}|v(t)|^2_{L_2(\Gamma)}=0, 
\]
and by \eqref{eq:korn}, we conclude that 
\begin{equation}\label{eq:gradvto0}
\lim_{t\to \infty}|v(t)|_{H^1_2(\Omega)}=0. 
\end{equation}
Choosing $\tau>t_0$, the first assertion follows from the fact that $L_{p,\mu}((t_0, T);X)|_{[\tau,T]}\subset L_p((\tau,T);X)$
for any $\tau \in (t_0,T)$. Here we also note that the range $q\in (1,2)$ is admissible, since $\Omega$ is bounded.

Let us conclude our proof by showing \eqref{eq:bvto0}.
Let $\alpha\in [0,1)$ and $q\in [2,6)$ be given. By choosing $p$ sufficiently large we have
$B^{2-2/p}_{qp}(\Omega)\hookrightarrow  H^{2\alpha}_q(\Omega).$ 
Let $p$ be fixed so that this embedding holds.
Moreover, the embedding $H^1_2(\Omega)\hookrightarrow B^{2\mu-2/p}_{qp}(\Omega)$ holds with $\mu =1/p+ 3/2q-1/4.$ 
Hence, \eqref{eq:gradvto0} implies 
\begin{equation}
\label{to0-inB}
\lim_{t\to \infty}\norm{v(t)}_{B^{2\mu-2/p}_{qp}}=0.
\end{equation}
We can now conclude from  \eqref{momentum-conserved} 
that $(v,{\sf a})([\tau,\infty))$ is relatively compact in 
\begin{equation*}
{_\parallel B}^{2\bar \mu-2/p}_{qp,\sigma}(\Omega)\times \R^3, \quad\text{for any } \bar\mu\in [\mu_{\rm crit},\mu), 
\end{equation*}
where $ \mu_{\rm crit}=1/p+3/2q-1/2.$
Theorem 5.7.1 in \cite{PrSi16} shows that $(v,{\sf a})([\tau+1,\infty))$ is 
compact, and hence also bounded, in $B^{2-2/p}_{qp}(\Omega)^3\times \R^3.$ 
Choosing $\theta$ such that
\[
( B^{2\mu-2/p}_{qp}(\Omega),B^{2-2/p}_{qp}(\Omega))_{\theta,2}= B^{2\alpha}_{q2}(\Omega)
\hookrightarrow H^{2\alpha}_q(\Omega)
\] 
yields \eqref{eq:bvto0} by interpolation.
In the case $q\in (1,2)$,  convergence in \eqref{eq:bvto0} follows from the above since $\Omega$ is a bounded domain.
\end{proof}
\section{Stability, instability, and long time behavior}
\label{se:long-time}
In this section, we analyze the stability properties of equilibria and the long time behavior of (weak) solutions to 
\eqref{problem}.
A precise knowledge of the spectrum of the linearization of \eqref{problem} at 
equilibria $(0,\sa_*)\in \cE$ is of paramount importance in our approach.

The linearization of  \eqref{problem} at a non-trivial equilibrium $(0,\sa_*)\in \cE$  is given by 
\begin{equation}
\label{linearize}
\begin{aligned}
\partial_t v + (\dot\sa -\dot\omega)\times x + 2({\sa_*}\times v) 
-\upnu \Delta v +\nabla p &=f_* &&\text{in} \;\; \Omega, \\
{\rm div}\, v &= 0 &&\text{in}\;\; \Omega, \\
(v|n) &=0 &&\text{on} \;\; \Gamma\, \\
2\upnu \,P_\Gamma(D(v)n)+\zeta v&=0 &&\text{on} \;\; \Gamma, \\
\II\, \dot\sa + \sa_*\times \II\sa + (\sa-\omega)\times\II\sa_* &=g_* &&\text{on}\;\; \R^3.\\
\end{aligned}
\end{equation}
Problem \eqref{linearize} can be written in the condensed form $\frac{d}{dt}u + L_*u=F_*(t)$,
where $L_*: X_1\to X_0$ is given by
\begin{equation*}  
\label{eq:L}
L_*\left[\begin{array}{l}
        v \\ 
        \vspace{-3mm} \\
        a
      \end{array}\right]
 = E^{-1}\left[\begin{array}{l}
        - \upnu\PP\Delta v  + 2 \PP(\sa_*\times v) \\
        \vspace{-3mm} \\
            \II\sa_*\times \II^{-1} \int_\Omega (x\times v)\,dx + \sa_* \times \II\sa   -\II\sa_* \times \sa
      \end{array}\right],  
\end{equation*} 
with $E$ defined in \eqref{def-E-A_0}.
Here we record that $L_*$ has compact resolvent, and hence its spectrum consists entirely of eigenvalues.
Associated to \eqref{linearize}, or equivalently, to $L_*$,  is the eigenvalue problem
\begin{equation}
\label{EV-problem}
\begin{aligned}
z\big( v + \PP((\sa -\omega)\times x)\big) + 2\PP({\sa_*}\times v) - \upnu\PP \Delta v &=0 &&\text{in}\;\; \Omega, \\
(v|n) &=0 &&\text{on} \;\; \Gamma\, \\
2\upnu \,P_\Gamma(D(v)n)+\zeta v&=0 &&\text{on} \;\; \Gamma, \\
z\II\sa + \sa_*\times \II\sa + (\sa-\omega)\times\II\sa_* &=0 &&\text{on}\;\; \R^3.\\
\end{aligned}
\end{equation}
Let $e_*=(0,\sa_*)\in\cE$ be a non-trivial equilibrium of \eqref{problem}. 
Then $e_*$ is called {\em normally stable} if
\begin{itemize}
\item[(i)] near $e_*$ the set of equilibria $\cE$ is a $C^1$-manifold in $X_1$,
\item[(ii)] \, the tangent space of $\cE$ at $u_*$ is given by $N(L_*)$,
\item[(iii)] \, $0$ is a semi-simple eigenvalue of $L_*$, i.e.\ $ N(L_*)\oplus R(L_*)=X_0$,
\item[(iv)] \, $\sigma(-L_*)\setminus\{0\}\subset [{\rm Re}\, z<0]$.
\end{itemize}
Moreover, $e_*$ is called {\em normally hyperbolic} if (i)--(iii) in the definition of normally stable hold, while (iv) is replaced by
\begin{itemize}
\item[(iv')]  $\sigma(L_*)\cap i\R =\{0\}$, $\sigma_u:=\sigma(-L_*)\cap [{\rm Re}\, z>0]\neq\emptyset$.
\end{itemize}
The next result states that each non-trivial equilibrium is either normally stable or normally hyperbolic.
\begin{theorem} 
\label{th:spectrum}
Let $(0,\sa_*)\in \cE$ be a non-trivial equilibrium of system~\eqref{problem}.

Then $(0,\sa_*)$ is normally stable if $\lambda_*=\max\{\lambda_1,\lambda_2,\lambda_3\}$,
and normally hyperbolic otherwise.

More precisely, 
assuming that the eigenvalues $\lambda_j$ of $\II$ are ordered by $\lambda_1\le \lambda_2\le\lambda_3$, we have
\begin{enumerate}
\item[{\bf (a)}] 
$-L_*$ has exactly one positive eigenvalue if $\lambda_1\le \lambda_\ast=\lambda_2<\lambda_3$.
\vspace{1mm}
\item[{\bf (b)}]
$-L_*$ has exactly two eigenvalues in $[{\rm Re}\,z>0]$ if $\lambda_*=\lambda_1<\lambda_2\le \lambda_3$.
\end{enumerate}
\end{theorem}
\begin{proof}
A careful analysis shows that the proof of \cite[Theorem 4.2]{MaPrSi18} applies to the case with 
Navier boundary conditions as well. In fact, it suffices to observe that the term
$\upnu |\nabla v|_\Omega$, occurring  from partial integration of
$(-\upnu \PP\Delta v| v)_\Omega$  in \cite{MaPrSi18},
corresponds to $2\upnu |D(v)|^2_\Omega + \zeta |v|^2_\Gamma$ in the current situation.
We note that  $2\upnu |D(v)|^2_\Omega + \zeta |v|^2_\Gamma =0$ implies $v=0$ by Korn's inequality,
which parallels the implication  $v=0$ in case $|\nabla v|_\Omega=0$ 
used  repeatedly in the proof of \cite[Theorem 4.2]{MaPrSi18}.
\end{proof}

In the next theorem,  we provide a characterization of the nonlinear stability properties of the non-trivial equilibria for 
problem~\eqref{problem}.
\begin{theorem}\label{th:stability-instability}
Suppose $p,q$ and $\mu$ satisfy the assumptions of Theorem~\ref{th:local_strong}
or of Remark~\ref{remark-2}(b). 
Let $(0,\sa_*)\in \cE$  be a non-trivial equilibrium of~\eqref{problem},
and recall that $\sa_*\in {\sf N}(\lambda_*-\II)$. 
Then the following statements hold. 
\begin{itemize}
\setlength\itemsep{1mm}
\item[(i)] If $\lambda_*=\max\{\lambda_1,\lambda_2,\lambda_3\}$, then $(0,\sa_*)$ is  stable
in ${_\parallel B}^{2\mu-2/p}_{qp,\sigma}(\Omega)\times \R^3$. 

\noindent
Moreover, there exists $\delta>0$ such that the unique solution $(v(t),\sa(t))$ of \eqref{eq:evolution} with initial value $(v_0,\sa_0) \in {_\parallel B}^{2\mu-2/p}_{qp,\sigma}(\Omega)\times \R^3$ satisfying 
\[
|(v_0, \sa_0-\sa_*)|_{B^{2\mu-2/p}_{qp}(\Omega)\times\R^3}<\delta
\]
exists on $\R_+$ and converges exponentially fast 
to some $(0,\bar\sa)\in \cE$
in the topology of  
$H_q^{2\alpha}(\Omega)^3\times \R^3$  for any fixed $\alpha\in [0,1)$.
\item[(ii)] If $\lambda_*\ne\max\{\lambda_1,\lambda_2,\lambda_3\}$, then $(0,\sa_*)$ is unstable in 
${_\parallel B}^{2\mu-2/p}_{qp,\sigma}(\Omega)\times \R^3.$ 
\end{itemize}
\end{theorem}
\begin{proof}
The result follows from  Theorem~\ref{th:spectrum},  the {\em generalized principle of linearized stability},
(see \cite[Chapter 5]{PrSi16} or \cite{PrSiZa09}), and parabolic regularization.
We refer to the proof of Theorem~5.2 in \cite{MaPrSi18} for technical details.
\end{proof}
\noindent
We are now ready state our main result concerning the long-time behavior of solutions to \eqref{problem}. 
\begin{theorem}\label{th:convergence}
Let $(v,\sa)$ be a weak solution as in  Theorem~\ref{th:weak},
corresponding to  an initial condition $(v_0,\sa_0)\in L_{2,\sigma}(\Omega)\times \R^3$ with $\sa_0\ne 0$.
Then there exists $\bar \sa\in \R^3$ with $ |\II \bar\sa| = |\II \sa_0|$, such that 
\begin{equation*}
(v(t),\sa(t))\to (0, \bar\sa)\quad\text{in} \;\; H^{2\alpha}_q(\Omega)^3\times \R^3, \;\; 
\text{for any fixed $\alpha\in[0,1)$ and  $q\in (1,6)$}, 
\end{equation*}
at an exponential rate.
\end{theorem}
\begin{proof}
We recall that, by Proposition \ref{prop:equilibria}(b),  ${\sf E}$  is a strict Lyapunov function for \eqref{eq:evolution}. 
Moreover, the nonzero equilibria have been characterized in Theorem \ref{th:spectrum} to be either normally stable or normally hyperbolic. 
Let  $\alpha\in (0,1)$ and $q\in (6/5,6)$ be fixed and choose $p$ large enough so that
$$ B^{2-2/p}_{qp}(\Omega)\hookrightarrow H^{2\beta}_q(\Omega)\cap H^{2\alpha}_q(\Omega),
\quad \beta <1-1/p,$$ 
where $\beta$  is as in  Remark~\ref{remark-2}(b). Therefore, 
$X_{\gamma,1}  \hookrightarrow X_\beta  \hookrightarrow X_{\gamma,\mu}$, 
and this implies $F\in C^{1}(X_{\gamma,1},X_0)$, where  $F$  is the nonlinearity of~\eqref{eq:evolution}.
Here we recall that $X_{\gamma,\nu}:=(X_0,X_1)_{\nu-1/p,p}$, and 
we also recall that $L$ has the property of maximal $L_p$-regularity. 
By Theorem \ref{th:weak-strong}, we know that there exists a time $\tau>0$ such that
$v\in BC([\tau,\infty); {_\parallel B}^{2-2/p}_{pq,\sigma}(\Omega))$.
This property in conjunction with ~\eqref{momentum-conserved}
implies relative compactness of the trajectory
$\{(v(t), \sa(t)): t\ge \tau\}$
in $X_{\gamma,\mu}.$
Exponential convergence to an equilibrium, in the stated topology, now follows from 
\cite{PrSi16}, Theorems~5.7.2, 5.3.1, 5.5.1, and the above embedding. 
It remains to observe that convergence in $H^{2\alpha}_q(\Omega)^3$ for $q\in (1,6/5]$ follows from 
the fact that  $\Omega$ is bounded.
\end{proof}

\begin{remark}{\rm
Suppose $\sa_0=0$. By Theorem \ref{th:weak-strong}, any weak solution $(v(t),\sa(t))$ corresponding to $(v_0,0)$ exists globally in time.  
Moreover, by conservation of total angular momentum \eqref{momentum-conserved}, necessarily $\sa(t)= 0$ for all times. 
Then, again by Theorem~\ref{th:weak-strong}, $v(t)$ is a strong solution of 
\[
\dot v+A_1v=f(v),\quad v(0)=v(\tau),
\]
with  $A_1v=-\upnu (I+C)\PP\Delta v$  and $f(v)=(I+C)\PP(-v\cdot \nabla v+2\;\omega\times v)$. 
By Proposition~\ref{pro:L-calculus}(b) and \cite[Corollary 2.2 (iii)]{PrSiWi18} it follows that the rate of convergence 
in \eqref{eq:bvto0} is in fact exponential. }
\end{remark}

{\bf Acknowledgment:}
This paper is dedicated to the memory of Jan Pr\"uss who passed away in July of 2018,
shortly before this joint manuscript was completed.

\end{document}